\documentclass{article}
\usepackage{amsmath}
\usepackage{amsthm}
\usepackage{amsfonts}
\usepackage{tcolorbox}
\usepackage{graphicx}

\usepackage{authblk}

\title{On the nontrivial extremal eigenvalues of graphs}
\author[1]{Wenbo Li}\author[2]{Shiping Liu}
\affil[1,2]{School of Mathematical Sciences, University of Science and Technology of China, Hefei 230026, China}
\affil[1]{patlee@mail.ustc.edu.cn}
\affil[2]{spliu@ustc.edu.cn}
\date{}
\begin{document}
\maketitle

\newtheorem{theorem}{Theorem}
\newtheorem{proposition}{Proposition}
\newtheorem{corollary}{Corollary}
\newtheorem{lemma}{Lemma}
\newtheorem{definition}{Definition}
\newtheorem{remark}{Remark}
\newtheorem{eg}{Example}

\begin{abstract}
We present a finer quantitative version of an observation due to Breuillard, Green, Guralnick and Tao
which tells that for finite non-bipartite Cayley graphs, once the nontrivial
eigenvalues of their normalized adjacency matrices are uniformly bounded away from $1$, then they are also uniformly bounded away from $-1$. Unlike previous works which depend heavily on combinatorial arguments, we rely more on analysis of eigenfunctions. We establish a new explicit lower bound for the gap between $-1$ and the smallest normalized adjacency eigenvalue, which improves previous lower bounds in terms of edge-expansion, and is comparable to the best known lower bound in terms of vertex-expansion.
\end{abstract}

\section{Introduction}

One of the main topics of spectral graph theory is to explore the relationship between structural properties of a graph and eigenvalues of associated matrices. Let $G=(V,E)$ be a finite graph with $n$ vertices. Denote by $$\mu_{n} \leq \mu_{n-1} \leq ... \leq \mu_{2} \leq \mu_{1}$$ the eigenvalues of its normalized adjacency matrix. We call an eigenvalue trivial if it equals $1$ or $-1$. Recall that $\mu_1$ is always equal to $1$ and $\mu_n=-1$ if and only if the graph has a bipartite connected component. The well-known Cheeger inequality \cite{alon1985lambda1,alon1986eigenvalues,dodziuk1984difference} $h^2/2\leq 1-\mu_2\leq 2h$ relates the spectral gap $1-\mu_2$ and the edge-expansion (also called Cheeger constant) $h$. In this article we show that the following inequality involving $\mu_2$, $\mu_{n-1}$ and $h$ holds.
\begin{theorem}\label{25}
    Let $G=(V,E)$ be a finite connected graph. Then we have
    \begin{equation}\label{11}
        1+\mu_{n-1}\geq \frac{(1-\mu_2)^2}{2h^2}
\left(\sqrt{1+\frac{h^2}{1-\mu_2}}-1\right)^2.
    \end{equation}
\end{theorem}

As an application, we show the following estimates for the spectral gap $1+\mu_n$ of a non-bipartite vertex-transitive graph.
\begin{corollary}\label{26}
    Let $G$ be a finite, non-bipartite, vertex-transitive graph.
    There hold
\begin{equation}\label{22}
    1+\mu_n \geq {\rm min} \left\{\frac{2}{d}, \frac{(\sqrt{3}-1)^2}{8}h^2\right\},
\end{equation}
and
\begin{equation}\label{23}
    1+\mu_n \geq {\rm min} \left\{\frac{2}{d}, 2\left(\sqrt{1+\frac{1-\mu_2}{4}}-1\right)^2 \right \}.
\end{equation}
\end{corollary}
This provides a finer version of an observation due to Breuillard-Green-Guralnick-Tao \cite[Proposition E.1]{1}, which tells that for finite non-bipartite Cayley graphs, combinatorial expansion implies spectral expansion. More precisely, they show that for such a connected graph, there exits $\delta>0$ depending only on its degree $d$ and vertex-expansion $h_{out}$, such that the following holds
\begin{equation}\label{9}
    1+\mu_n \geq \delta(d,h_{out}).
\end{equation}
Here the vertex-expansion $h_{out}$ is closely related to the edge-expansion $h$. Indeed, for $d$-regular graphs there holds $ h_{out}/d \leq h \leq h_{out}.$

By Cheeger inequality, their observation simply tells that, for finite non-bipartite Cayley graphs, once the gap $1-\mu_2$ is uniformly bounded away from $0$, so does the gap $1+\mu_n$.

It is natural to seek for an explicit formula for the lower bound
of $1+\mu_n$ in terms of $d$, $h_{out}$ or other closely related constants like the edge-expansion $h$ for Cayley, or more generally, vertex-transitive graphs. Various works have been done on this topic, see, for example, \cite{biswas2019cheeger,biswas2021cheeger,biswas2021spectral,moorman2022bipartiteness,saha2023cheeger,hu2023vertex}. A recent result of Saha \cite{saha2023cheeger} states that for non-bipartite vertex-transitive graphs there holds
\begin{equation}\label{20}
    1+\mu_n \geq C\frac{h^2}{d^2}.
\end{equation}
We use $C$ to denote an absolute constant which may change from line to line.
Indeed, we have by the dual Cheeger inequality \cite{trevisan2009max,bauer2009bipartite} $1+\mu_n\geq \beta^2/2$, where $\beta$ is the bipartiteness constant of the graph $G$. Saha \cite{saha2023cheeger} proves for non-bipartite vertex-transitive graphs that $\beta\geq C h/d$, solving an open question of Moorman-Ralli-Tetali \cite{moorman2022bipartiteness}.
Extending the work of Bobkov-Houdr\'e-Tetali \cite{bobkov2000lambda} on vertex-expansion to the setting of signed graphs,  Hu-Liu \cite{hu2023vertex} establish that for non-bipartite vertex-transitive graphs
\begin{equation}\label{21}
    1+\mu_n \geq C\frac{h_{out}^2}{d}.
\end{equation}
Recalling $h_{out}\geq h$, the estimate $(\ref{21})$ improves (\ref{20}).
Notice that our estimates (\ref{22}) and (\ref{23}) cannot be derived from (\ref{21}) and vice versa.

We achieve our estimates via a very different strategy from previous works \cite{biswas2019cheeger,biswas2021cheeger,biswas2021spectral,moorman2022bipartiteness,saha2023cheeger,hu2023vertex}. We rely more on spectral graph theoretic methods than combinatorial arguments. We are motivated by the Remark in \cite[Appendix E]{1} to consider the multiplicity of the eigenvalue $\mu_n$ and the product of two eigenfunctions, and find \cite{desai1994characterization} inspiring in obtaining Lemma \ref{3}.

Our strategy can be described as follows. By Theorem \ref{25}, all eigenvalues of a finite connected graph with multiplicity greater than $1$ is bounded away from $-1$.
On the other hand, for non-bipartite vertex-transitive graphs there is a natural gap of at least $2/d$ between any simple eigenvalue and $-1$, due to the symmetry of the graph. Combining the above two cases leads to the bounds in (\ref{22}) and (\ref{23}).

\section{Preliminaries}
Let $G = (V, E) $ be a finite graph. Recall that the degree matrix $D$ and adjacency matrix $A$ of $G$ is defined as follows:
\begin{equation*}
    D_{ij}=
    \begin{cases}
      d_i&  i=j,\\
       0  & i\neq j,
    \end{cases}
\end{equation*}
where $d_i$ is the degree of $i\in V$ and
\begin{equation*}
    A_{ij}=
    \begin{cases}
      1 &  \{i,j\} \in E,\\
       0  & \{i,j\} \notin E.
    \end{cases}
\end{equation*}
The \emph{normalized adjacency matrix} of a graph is defined as $D^{-1}A$.

Let $\mathbb{R}^{V}$ be the set of real-valued functions on $V$. For any $f,g\in \mathbb{R}^{V}$, we define their inner product as \[\langle f, g\rangle:=\sum_{u\in V}f(u)g(u)d_u.\] The corresponding $\ell^2$-norm of a function $f$ is defined as $\Vert f\Vert_2:=\langle f,f\rangle^{\frac{1}{2}}$. The following two identities are straightforward.
\begin{lemma}\label{1}
Let $G=(V,E)$ be a finite graph and $f$ be a function on $V$. We have
\begin{equation}
    \sum_{\{u,v\}\in E}(f(u)-f(v))^2=\  \langle f,(I-D^{-1}A)f\rangle
\end{equation}
and
\begin{equation}
    \sum_{\{u,v\}\in E}(f(u)+f(v))^2=\ \langle f,(I+D^{-1}A)f \rangle,
\end{equation}
where $I$ is the $n\times n$ identity matrix.
\end{lemma}

We list the $n:=|V|$ eigenvalues of $D^{-1}A$ counting multiplicity as
$$\mu_{n} \leq \mu_{n-1} \leq ... \leq \mu_{2} \leq \mu_{1}.$$
In fact, $\mu_{2}=1$ if and only if $G$ is disconnected, and $\mu_{n}=-1$ if and only if G has a bipartite connected component.

For any subset $S \subset V$, we define its volume vol$(S)$ as $${\rm vol}(S)=\sum_{i\in S} d_i,$$ and
its boundary $\partial S$ as
\begin{equation*}
    \partial S=\left\{\{i,j\}\in E:i \in S, j\notin S \right\}.
\end{equation*}
The edge-expansion of $S$ is then defined to be
\begin{equation}
    h(S)=\frac{|\partial S|}{{\rm min}\left({\rm vol}(S),{\rm vol}(V-S)\right)}.
\end{equation}
The edge-expansion (also called Cheeger constant) of $G$, denoted by $h$, is defined as
\begin{equation}
    h:=\mathop{\rm min}\limits_{S \subset V}h(S).
\end{equation}
Note that $0 \leq h \leq 1$, and $h=0$ if and only if $G$ is disconnected. Let us recall two classic results involving $h$.

\begin{proposition}[{\cite[Corollary 2.10]{chung1997spectral}}]\label{4}
Let $G=(V,E)$ be a finite graph. For any nonzero function $f$ on $V$ satisfying
$\sum_{u \in V}f(u)d_u=0$, we have
\begin{equation}
    \frac{\sum_{\{u,v\}\in E}|f(u)-f(v)|}{\sum_{u \in V}|f(u)|d_u}\geq \frac{1}{2}h.
\end{equation}

\end{proposition}

\begin{theorem}[Cheeger inequality \cite{alon1985lambda1,alon1986eigenvalues,dodziuk1984difference}]
Let $G=(V,E)$ be a finite graph. We have
\begin{equation}\label{18}
    2h \geq 1-\mu_2 \geq \frac{1}{2}h^2.
\end{equation}
\end{theorem}

\section{Proof of Theorem \ref{25}}

First, we prepare the following lemmas.

\begin{lemma}\label{17}
Let $G=(V,E)$ be a finite graph and $f,g$ be two eigenfunctions of the eigenvalues $\mu$ and $\nu$ of the normalized adjacency matrix $D^{-1}A$, respectively.
We have the following estimate
\begin{equation*}
    \sum_{\{u,v\}\in E}|f(u)g(u)-f(v)g(v)|\leq \frac{\sqrt{2}}{2} \left(\sqrt{1+\mu}+\sqrt{1+\nu}\right)\Vert f\Vert_2\Vert g\Vert_2.
\end{equation*}
\begin{proof}
By Cauchy-Schwarz inequality and Lemma \ref{1}, we calculate
\begin{align*}
    & \sum_{\{u,v\}\in E}2|f(u)g(u)-f(v)g(v)|\\  = &\sum_{\{u,v\}\in E}|(f(u)-f(v))(g(u)+g(v))+(f(u)+f(v))(g(u)-g(v))|\\
    \leq & \sum_{\{u,v\}\in E}|(f(u)-f(v))(g(u)+g(v))|+\sum_{\{u,v\}\in E}|(f(u)+f(v))(g(u)-g(v))|\\
    \leq &\sqrt{\sum_{\{u,v\}\in E}(f(u)-f(v))^2}\sqrt{\sum_{\{u,v\}\in E}(g(u)+g(v))^2} \\
    &\hspace{3.5cm} +\sqrt{\sum_{\{u,v\}\in E}(f(u)+f(v))^2}\sqrt{\sum_{\{u,v\}\in E}(g(u)-g(v))^2}\\
    =& \left(\sqrt{(1-\mu)(1+\nu)}+\sqrt{(1+\mu)(1-\nu)}\right) \Vert f\Vert_2\Vert g\Vert_2 \\
    \leq & \sqrt{2} (\sqrt{1+\mu}+\sqrt{1+\nu})\Vert f\Vert_2\Vert g\Vert_2.
\end{align*}
This completes the proof.
\end{proof}
\end{lemma}
Lemma \ref{17} tells that the $\ell^1$-energy of the product $fg$ of two eigenfunctions is close to $0$, whenever the corresponding two eigenvalues are close to $-1$.

The next lemma shows that if an eigenvalue is close to $-1$, then the absolute value of its eigenfunction is close to being a constant.
\begin{lemma}\label{3}
    Let $G=(V,E)$ be a finite connected graph. Assume that $f$ is an eigenfunction of the eigenvalue $\mu$ of the normalized adjacency matrix $D^{-1}A$ with $\Vert f\Vert_2=1$.
    Let  $\boldsymbol{c}$ is the constant function such that $\Vert c\Vert_2=1$ and $f_1:=|f|-\langle |f|,\boldsymbol{c}\rangle\boldsymbol{c}$. Then we have

\begin{equation*}
   \Vert f_1\Vert_2^2  \leq \frac{1+\mu}{1-\mu_2}
\end{equation*}
and
\begin{equation*}
    \langle |f|, \boldsymbol{c}\rangle^2 \geq 1-\frac{1+\mu}{1-\mu_2}.
\end{equation*}
\end{lemma}
\begin{proof}
Oberserving that $\langle f_1, \boldsymbol{c} \rangle=0$, we have $$(1-\mu_2) \langle f_1,f_1\rangle \leq \langle f_1,(I-D^{-1}A)f_1\rangle.$$
Therefore, we estimate by Lemma \ref{1}
\begin{equation*}
    \begin{split}
       \Vert f_1\Vert_2^2 &\leq \frac{1}{1-\mu_2}\langle f_1,(I-D^{-1}A)f_1\rangle= \frac{1}{1-\mu_2}\langle |f|,(I-D^{-1}A)|f|\rangle\\
       &= \frac{1}{1-\mu_2} \sum_{\{u,v\}\in E}(|f(u)|-|f(v)|)^2\\
       & \leq \frac{1}{1-\mu_2} \sum_{\{u,v\}\in E}(f(u)+f(v))^2\\
       & =\frac{1}{1-\mu_2} \langle f,(I+D^{-1}A)f \rangle\\
       & =\frac{1+\mu}{1-\mu_2}.
       \end{split}
\end{equation*}
The other inequality holds since $$1=\Vert f\Vert_2^2=\langle |f|,\boldsymbol{c}\rangle^2+\Vert f_1\Vert_2^2.$$
This concludes the proof.
\end{proof}
The following lemma is a useful consequence of Lemma \ref{3}. It bounds from below the $\ell^1$-norm of the product of two eigenfunctions.
\begin{lemma} \label{5}
    Let $G=(V,E)$ be a finite connected graph and $f,g$ be eigenfunctions of the eigenvalues $\mu$ and $\nu$ of $D^{-1}A$, respectively, such that  $\Vert f\Vert_2 =\Vert g\Vert_2=1$. Then
    \begin{equation}
        \langle |f|,|g|\rangle \geq \sqrt{1-\frac{1+\mu}{1-\mu_2}}\sqrt{1-\frac{1+\nu}{1-\mu_2}}-\sqrt{\frac{1+\mu}{1-\mu_2}}\sqrt{\frac{1+\nu}{1-\mu_2}}.
    \end{equation}
\begin{proof}
    Let  $\boldsymbol{c}$ be a positive constant function with $\Vert c\Vert_2=1$. Decompose $|f|$ and $|g|$ such that
$$|f|=\langle |f|,\boldsymbol{c}\rangle\boldsymbol{c}+f_1,\,\,\,\,\text{and}\,\,\,\,|g|=\langle |g|,\boldsymbol{c}\rangle\boldsymbol{c}+g_1.$$ Then we compute
\begin{equation*}
\begin{split}
    \langle |f|,|g|\rangle&=\langle |f|,\boldsymbol{c}\rangle\langle |g|,\boldsymbol{c}\rangle+\langle f_1,g_1\rangle\\
    &\geq \langle |f|,\boldsymbol{c}\rangle\langle |g|,\boldsymbol{c}\rangle-\Vert f_1\Vert_2\Vert g_1\Vert_2\\
    &\geq\sqrt{1-\frac{1+\mu}{1-\mu_2}}\sqrt{1-\frac{1+\nu}{1-\mu_2}}-\sqrt{\frac{1+\mu}{1-\mu_2}}\sqrt{\frac{1+\nu}{1-\mu_2}}
\end{split}
\end{equation*}
where the last inequality comes from Lemma \ref{3}.
\end{proof}
\end{lemma}
We are now prepared to prove Theorem \ref{25}.
\begin{proof}[Proof of Theorem \ref{25}]
    Let $f,g$ be eigenfunctions of the eigenvalues $\mu_n$ and $\mu_{n-1}$ such that $\Vert f\Vert_2=\Vert g\Vert_2=1$ and
    $\langle f, g\rangle=0$, i.e., $\sum_{u\in V}f(u)g(u)d_u=0$. Applying Proposition \ref{4} to the function $fg$ and inserting the estimates from Lemma \ref{17} and Lemma \ref{5}, we derive the following inequality
    $$
    \frac{\sqrt{2}}{2}h\leq \frac{ \sqrt{1+\mu_n}+\sqrt{1+\mu_{n-1}}}{\sqrt{1-\frac{1+\mu_n}{1-\mu_2}}\sqrt{1-\frac{1+\mu_{n-1}}{1-\mu_2}}-\sqrt{\frac{1+\mu_n}{1-\mu_2}}\sqrt{\frac{1+\mu_{n-1}}{1-\mu_2}}},
    $$
    when $1+\mu_{n-1} < (1-\mu_2)/2$.
    Notice that the following function
\begin{equation}
f(x,y)=\frac{\sqrt{x}+\sqrt{y}}{\sqrt{1-\frac{1}{1-\mu_2}x}\sqrt{1-\frac{1}{1-\mu_2}y}-\sqrt{\frac{1}{1-\mu_2}x}\sqrt{\frac{1}{1-\mu_2}y}}
\end{equation}
is continuous on $[0,\frac{1}{2}(1-\mu_2))\times[0,\frac{1}{2}(1-\mu_2))$, monotonically increasing with respect to either $x$ or $y$ and $f(0,0)=0$. Let $c=c(h,1-\mu_2)$ be the number such that $$f(c,c)=\frac{\sqrt{2}}{2}h.$$ Then we have $$1+\mu_{n-1}\geq c,$$ since, otherwise,
$$f(1+\mu_n,1+\mu_{n-1})< f(c,c)=\frac{\sqrt{2}}{2}h$$ is a contradiction. A direct computation shows that
\begin{equation*}
        \sqrt{c}=\frac{1-\mu_2}{\sqrt{2}h}
\left(\sqrt{1+\frac{h^2}{1-\mu_2}}-1\right).
\end{equation*}
This proves the inequality (\ref{11}) when $1+\mu_{n-1}<(1-\mu_2)/2$.

For the case that $1+\mu_{n-1} \geq (1-\mu_2)/2$, the inequality (\ref{11}) still holds since
$$1+\mu_{n-1} \geq \frac{1-\mu_2}{2}\geq \frac{(1-\mu_2)^2}{2h^2}
\left(\sqrt{1+\frac{h^2}{1-\mu_2}}-1\right)^2,$$
while the last inequality comes from the fact that $1-\mu_2 \geq 0$.
\end{proof}

It is direct to check that $c=c(h,1-\mu_2)$ is  monotonically increasing with respect to either $h$ or $1-\mu_2$. Replacing $1-\mu_2$ with powers of $h$ or doing the opposite via the Cheeger inequality (\ref{18}) yields the following result.

\begin{corollary}\label{15}
    Let $G=(V,E)$ be a finite graph. Let $n$, $\mu_2$, $\mu_{n-1}$ and  $h$ be as in Theorem \ref{1}. Then, we have
    \begin{equation}
        1+\mu_{n-1}\geq 2\left(\sqrt{1+\frac{1-\mu_2}{4}}-1\right)^2,
    \end{equation}

\begin{equation}\label{10}
        1+\mu_{n-1}\geq \frac{(\sqrt{3}-1)^2}{8}h^2.
\end{equation}
\end{corollary}
\begin{remark}
    If we do not care about the constant, the inequality (\ref{10}) can also be derived from the higher order dual Cheeger inequalities  \cite[Theorem 1.2]{liu2015multi}. Indeed, we have $1+\mu_{n-1}\geq C(1-\overline{h}(2))^2$, where $\overline{h}(2)$ stands for the two-way dual Cheeger constant. Then the inequality (\ref{10}) follows directly from the observation that $1-\overline{h}(2)\geq h$. The proof of higher order dual Cheeger inequalities involves applying deep results from the random partition theory. Our proof here is much simpler and more elementary. We also obtain a better constant here.
\end{remark}

\section{On non-bipartite vertex-transitive graphs}
For graphs such that $\mu_n=\mu_{n-1}$, Theorem \ref{1} provides a bound for the smallest eigenvalue of the normalized adjacency matrix. For vertex-transitive graphs we have the following gap phenomenon.
\begin{theorem}\label{thm:vertextransitive}
    Let $G$ be a finite connected vertex-transitive graph and $\mu$ a simple eigenvalue of its normalized adjacency matrix. Then
    \begin{equation}
    \mu=\frac{2k}{d}-1
    \end{equation}
    where $0\leq k \leq d$ is an integer and
    $d$ is the degree of $G$.
\end{theorem}
\begin{proof}
Fix $a\in V$. Suppose that $\mu$ is a simple eigenvalue of $D^{-1}A$ with $f$ being its eigenfunction. Assume further that $f(a)=-1$ or $1$. For every $g \in \mathrm{Aut}(G)$, the function $f_g$ defined as $f_g(\cdot)=f(g(\cdot))$ is still an eigenfunction of the same eigenvalue $\mu$. Since $\mu$ is simple, there exits $\lambda_g$ such that $f_g=\lambda_g f$. Since $\Vert f\Vert_2=\Vert f_g\Vert_2$, we have $\lambda_g \in \{-1,1\}$. From the fact that $G$ is vertex transitive, $\mathrm{Aut}(G)$ acts transitively on $G$, we have for any $u\in V$ that $f(u)\in \{-1,1\}$. Calculating $D^{-1}Af(a)$ yields the result.
\end{proof}

Combining Corollary \ref{15} and the above Theorem \ref{thm:vertextransitive}, we prove Corollary \ref{26}.

For graphs with no simple eigenvalues, we can get rid of the term $2/d$ in (\ref{22}) and (\ref{23}).

\begin{corollary}\label{16}
    Let $(G,S)$ be a finite connected Cayley graph with a finite group $G$ and a gernerating set $S$. Assume that $G$ is a simple group or the size $|G|$ is odd. Let $$-1< \mu_{n} \leq \mu_{n-1} \leq ... \leq \mu_{2} \leq \mu_{1}=1$$ be the eigenvalues of its normalized adjacency matrix where $n=|G|$. Denote by $h$ its edge-expansion. Then we have
\begin{equation}\label{7B}
    1+\mu_n \geq \frac{(\sqrt{3}-1)^2}{8}h^2
\end{equation}
and
\begin{equation}
        1+\mu_{n}\geq  2\left(\sqrt{1+\frac{1-\mu_2}{4}}-1\right)^2.
    \end{equation}
\end{corollary}
\begin{proof}
    The map $\lambda_g: G \rightarrow \{\pm 1\}$ in the proof of Theorem \ref{thm:vertextransitive} is in fact a homomorphism of groups since for any $g_1,g_2 \in G$, it holds that \[\lambda_{g_1g_2}f(e)=f_{g_1g_2}(e)=f(g_1g_2)=f_{g_1}(g_2)=\lambda_{g_1}f(g_2)=\lambda_{g_1}\lambda_{g_2}f(e).\]
    Here we denote by $e$ the identity element of the group $G$. If $G$ is a simple group or the size $|G|$ is odd, then such a homomorphism must be trivial otherwise there will be a subgroup of $G$ of index 2. As a result all non-trivial eigenvalues of $(G,S)$ will be of multiplicity great than 1 and the above estimates hold.
\end{proof}

\section*{Acknowledgement}
SL is very grateful to Paul Horn and Matthias Keller for very inspiring discussions on related topics.
This work is supported by the National Key R and D Program of China 2020YFA0713100, the National Natural Science Foundation of China (No. 12031017), and Innovation Program for Quantum Science and Technology 2021ZD0302902.

\bibliographystyle{abbrv}

\end{document}